\documentclass[12pt,twoside, draft]{article}
\usepackage{ifthen}
\usepackage{amsmath}
\usepackage{amsfonts}
\usepackage{latexsym}
\usepackage{amsthm}
\usepackage{amssymb}

\date{}

\begin{document}


\centerline{}

\centerline{}

\centerline {\Large{\bf Baire categories and classes of analytic functions }}

\centerline{\Large{\bf in which the
Wiman-Valiron type inequality }}

\centerline{\Large{\bf  can be almost surely improved}}

\centerline{}

\centerline{\bf {A. O. Kuryliak}}

\centerline{}

\centerline{Department of Mechanics and Mathematics,}

\centerline{Ivan Franko National University of L'viv, Ukraine}

\centerline{kurylyak88@gmail.ru}

\centerline{}

\centerline{\bf {O. B. Skaskiv}}

\centerline{}

\centerline{Department of Mechanics and Mathematics,}

\centerline{Ivan Franko National University of L'viv, Ukraine}

\centerline{matstud@franko.lviv.ua}

\centerline{}

\centerline{\bf {I. E. Chyzhykov}}

\centerline{}

\centerline{Department of Mechanics and Mathematics,}

\centerline{Ivan Franko National University of L'viv, Ukraine}

\centerline{chyzhykov@yahoo.com}

\newtheorem{Theorem}{\quad Theorem}[section]

\newtheorem{Definition}[Theorem]{\quad Definition}

\newtheorem{Corollary}[Theorem]{\quad Corollary}

\newtheorem{Lemma}[Theorem]{\quad Lemma}

\newtheorem{Example}[Theorem]{\quad Example}

\begin{abstract}
Let $f(z)=\sum_{n=0}^{+\infty} a_nz^n$\ $(z\in\mathbb{C})$\ be an analytic function in the unit disk and $f_t$ be an analytic function of the form
$f_t(z)=\sum_{n=0}^{+\infty} a_ne^{i\theta_nt}z^n,$ where $t\in\mathbb{R},$
$\theta_n\in\mathbb{N},$ and $h$ be a positive continuous function on $(0, 1)$  increasing
to $+\infty$ and such that $
\int_{r_0}^1h(r)dr=+\infty, r_0\in(0,1).$\ If the sequence
$(\theta_n)_{n\geq0}$ satisfies the inequality
$$
\varlimsup_{n\to+\infty}\frac1{\ln n}\ln\frac{\theta_n}{\theta_{n+1}-\theta_n}\leq\delta\in[0,1/2),
$$
then for all analytic functions $f_t$ almost surely for $t$ there exists a set $E=E(\delta,t)\subset(0,1)$
such that $\int_Eh(r)dr<+\infty$ and
$$
\varlimsup_{\begin{substack} {r\to1-0 \\ r\notin E}\end{substack}} \frac{\ln M_f(r,t)-\ln\mu_f(r)}{2\ln h(r)+\ln\ln\{h(r)\mu_f(r)\}}\leq\frac{1+2\delta}{4+3\delta},
$$
where $M_f(r,t)=\max\{|f_t(z)|\colon |z|=r\},$\ $\mu_f(r)=\max\{|a_n|r^n\colon n\geq 0\}$\ for $r\in[0, 1).$
\end{abstract}

{\bf Subject Classification:} 30B10, 30B20, 54E52 \\

{\bf Keywords:} random analytic functions, Wiman-Valiron's type inequality, Baire categories

\section{Introduction}

Let
 $H$ be the class of positive continuous functions on the interval $(0, 1)$ increasing to $+\infty$  and
 such that $
\int\nolimits_{r_0}^1h(r)dr=+\infty,\ r_0\in(0,1).
$

For a measurable set $E\subset (0, 1),$ the $h$-measure of $E$ is defined by
$$
h\mbox{-meas }(E)\overset{def}= \int_{E}h(r) dr,
$$
where $h\in H.$ It is clear that $ h\mbox{-meas}((0,1))=+\infty.$

Let $f$ be an analytic function  in the unit disc $\mathbb{D}=\{z\colon |z|<1\}$
of the form
\begin{equation}\label{1}
f(z)=\sum_{n=0}^{+\infty} a_nz^n.
\end{equation}
 For $r\in (0, 1)$  we denote  the maximum modulus of the function $f$ by $M_f(r)=\max\{|f(z)|\colon |z|=r\},$
 and the maximal term of the series  \eqref{1} by $\mu_f(r)=\break\max\{|a_n|r^n\colon n\geq0\}$.
Let also
\begin{gather*}
G_f(r)=\sum\limits_{n=0}^{+\infty}|a_n|r^n,\
  S_f(r)=\left(\sum_{n=0}^{+\infty}|a_n|^2r^{2n}\right)^{1/2},\\
 \Delta_{h}(r,f)=\frac{\ln M_f(r)-\ln\mu_f(r)}{2\ln h(r)+\ln_2\{h(r)\mu_f(r)\}},\\
 E(\eta,f,h)=\big\{r\in(0,1)\colon
  M_f(r)>\mu_f(r)(h^2(r)\ln\{h(r)\mu_f(r)\})^{\eta}\big\},
\end{gather*}
where $\ln_kx\overset{def}=\ln(\ln_{k-1} x)\ (k\geq 2),\
\ln_1x\overset{def}=\ln x.$

From the results proved in \cite{ko} it follows that in the case when
$h(r)=(1-r)^{-1},$ for every analytic function $f$ in $\mathbb{D}$ of the form (\ref{1}) there exists a set $E\subset (0, 1)$ of finite logarithmic measure, i.e. $ h\mbox{-meas}(E)<+\infty $, such that
\begin{equation}\label{2}
\varlimsup_{\begin{substack} {r\to1-0 \\ r\notin E}\end{substack}}\Delta_{h}(r,f)\leq\frac12.
\end{equation}

In \cite{skkur} the similar statement is proved with an arbitrary function $h\in H$ for which either $\ln h(r)=O(\ln_2 G_f(r))$ or
$\ln_2 G_f(r)=O(\ln h(r))$ $(r\to 1-0).$

In \cite{sul} it is noted that the constant $1/2$ in the inequality
\eqref{2} 
 cannot be replaced by a smaller number in general.
Indeed, if $g(z)=\sum_{n=1}^{+\infty}\exp\{\sqrt{n}\}z^n,$ then
for $ h(r)=(1-r)^{-1}$ we have
$$
\varliminf_{r\to1-0} \frac{M_g(r)}{ h(r) \mu_g(r)\ln^{1/2}\{\mu_g(r) h(r)\}}\geq C>0.
$$

In connection with this the following { question} arises naturally: {\it how can one describe the ``quantity''
of those analytic functions for which inequality~(\ref{2}) can
be improved?}

In the paper \cite{filk} it is proved that in some probability sense for ``majority'' of analytic functions
the constant $1/2$ in the inequality \eqref{2} can be replaced by $1/4.$ Similar statement is proved in \cite{skkur} in reference to the inequality
\eqref{2} with any function $h\in H$ described above.

At the same time, the classes
of random analytic functions considered in \cite{skkur}, \cite{filk} do not include all analytic functions
of the form
\begin{equation} \label{3}
f_t(z)=\sum_{n=0}^{+\infty}a_ne^{i\theta_n t}z^n,
\end{equation}
where $(\theta_n)_{n\geq0}$ is an arbitrary sequence of nonnegative integers.
Note that $f_0(z)\equiv f(z).$

 We suppose that the sequence $(\theta_n)_{n\geq0}$ satisfies the inequality
 \begin{equation}\label{4}
 \frac{\theta_{n+1}}{\theta_n}\geq q>1\ (n\geq0).
 \end{equation}
 In the case of $q\geq 2$ analytic functions  of the form  \eqref{4} satisfies the conditions of theorems from \cite{skkur}, \cite{filk} mentioned
 above.

We also remark that the possibility of  improvement of
Wiman-Valiron's inequality for entire functions of the form
(\ref{3}) was considered earlier by M.~Still \cite{st} and
P.~Filevych~\cite{filMS6} (see also \cite{skzr1}). A similar
question for the class of entire functions of two variables was
concidered in the papers \cite{skzr4}, \cite{skzr3} and
\cite{skzr2}. In  \cite{filBai}  the ``quantity''of  those entire
functions for which    classical Wiman-Valiron's inequality can be
improved, is described in the sense of Baire's categories.

Here we consider the formulated {\it question} in the class of
analytic functions in $\mathbb{D}$ of the form~(\ref{3}). The proved
theorems complement in this case theorems from \cite{filk, skkur}
and are analogues of the statements from \cite{st} and \cite{filBai}.

\section{Auxiliary lemmas}

We need Lemma 2 from \cite{st} (see also \cite{filMS6}).
\begin{Lemma}[\cite{st}] \label{l1}
If a sequence $(\theta_n)_{n\geq0}$ satisfies the condition (\ref{4}), then for all
sequences $(a_n)_{n\geq0},$ $a_n\in\mathbb{C},$ and all $\beta>0, N\geq 0$ we have
\begin{gather*}
P_0\Bigg(\Bigg\{t\in[0, 2\pi]\colon \max_{0\leq
\psi\leq2\pi}\Bigg|\sum_{k=0}^{N}
a_k e^{ik\psi}e^{i\theta_kt}\Bigg|
\geq A_{\beta q}S_N\ln^{1/2}N\Bigg\}\Bigg) \leq  \frac1{N^{\beta}},
\end{gather*}
where $A_{\beta q}$ is a constant which depends only on $\beta$ and $q,$
$S_N=\sum_{n=0}^N|a_n|^2,$ $P_0=\frac{\mathfrak{m}}{2\pi},$
$\mathfrak{m}$ is the Lebesgue measure on the real line.
\end{Lemma}
\begin{Lemma}[\cite{filk}] \label{l2}
Let $k(r)$ be a continuous increasing to $+\infty$ function on $(0,1)$ ,
 $E\subset(0,1)$ be an open set such that there exists a sequence
$0<p_{1}\le ...\break \le p_{n}\to 1$ $(n\to+\infty)$ outside $E.$ Then
there exists a sequence $0<r_{1}\le ...\break\le r_{n}\to 1$
$(n\to+\infty)$ such that for all  $n\in \mathbb N$
\begin{itemize}
\item[1)] $r_{n}\notin E,$

\item[2)] $\ln k(r_{n})\ge\frac{n}{2},$

\item[3)] if $(r_{n};r_{n+1})\cap E\not=(r_{n},r_{n+1}),$ then $k(r_{n+1})\le e
k(r_{n}).$
\end{itemize}
\end{Lemma}
\begin{Lemma}[\cite{skkur}] \label{l3}
Let $\varphi_1(x)$ and $\varphi_2(x)$ be  positive continuous
increasing to $+\infty$ functions on $[0,+\infty)$ such that
$\int_0^{+\infty}\frac{dx}{\varphi_i(x)}<+\infty$\ $(i\in\{1,
2\}),$ $h\in H$ and $g_1(x)=\ln G_f(e^x)$\ $(x<0).$ Then there
exists a set $E\subset(0,1)$ such that $h\mbox{-meas}(E)<+\infty$
and for all $r\in(0,1)\backslash
 E$ we get
$$
g''_1(\ln r)\leq h(r)\varphi_2(h(r)\varphi_1(g_1(\ln r))).
$$
\end{Lemma}

We also denote
\begin{gather*}
A(r)=g'_1(\ln r)=\frac{d\ln G_f(r)}{d\ln r}=
\sum_{n=0}^{+\infty}\frac{n|a_n|r^n}{ G_f(r)},\\
B^2(r)=g''_1(\ln r)=\sum_{n=0}^{+\infty}\frac{n^2|a_n|r^n}{ G_f(r)}-A^2(r).
\end{gather*}

\begin{Lemma} \label{l4}
For $h\in H$ and all $\varepsilon>0$ there exists a set $E\subset(0,1)$ such that $h\mbox{-meas}(E)<+\infty$ and for all $r\in(0,1)\backslash
 E$ we have
\begin{gather*}
A(r)\leq h(r)\ln\{h(r)\mu_f(r)\}\ln_2^{1+\varepsilon}\{h(r)\mu_f(r)\},\\
B^2(r)\leq h^{2+\varepsilon }(r)\ln\{h(r)\mu_f(r)\}\ln_2^{2+\varepsilon}\{h(r)\mu_f(r)\}.
\end{gather*}
\end{Lemma}
\begin{proof}[Proof.] Let $(\Omega, \mathcal{A}, P)$\ be
a probability space which contains the discrete random variable
$\xi$ with the distribution
$$
P(\xi=n)=\frac{|a_n|e^{nx}}{G_f(e^x)} .
$$
Then the mean $M\xi=g'_1(x)$ and the variance $D\xi=g''_1(x).$

  Let $x=\ln r<0.$ Using Chebyshev's inequality we get
  $P(|\xi-g'_1(x)|<\sqrt{2g''_1(x)}\ )\geq 1/2,$ i.e.
  \begin{gather}
  \nonumber
g(x)\le 2\sum_{|n-g_1'(x)|<\sqrt{2g_1''(x)}}|a_n|e^{xn}\leq\\
\label{star}
\leq2\mu_{f}(r)\sum_{|n-g_1'(x)|<\sqrt{2g_1''(x)}}1
\leq2\mu_{f}(r)(2\sqrt{2g_{1}{''}(x)}+1).
  \end{gather}

For fixed $\varepsilon_1>0,\ \varepsilon_2>0$ we define
\begin{gather*}
E_1=\{x<0\colon g_{1}''(x)>h(e^x)g_{1}{'}(x)
(\ln g_{1}'(x))^{1+\varepsilon_1},\ g_{1}'(x)\ge 2\},\\
E_2=\{x<0\colon g_1'(x)>h(e^x)g_1(x)(\ln g_1(x))^{1+\varepsilon_2},\ g_{1}(x)\ge 2\}.
\end{gather*}

So,
$$
\int_{E_1\bigcup E_2}{h(e^x) dx}=\int_{E} \frac{h(r)}{r}dr<+\infty,\
\int_{E}h(r)dr<+\infty,
$$
where $E$ is the image of the set $E_1\cup E_2$ by the mapping $r=e^x.$
Therefore, $h$-$\mbox{meas} E=\int_{E} h(r)dr<+\infty.$

Then from (\ref{star}) we obtain as $r\to1-0,\ (r\notin E)$
\begin{gather*}
g(\ln r)\leq
2\mu_f(r) \Bigl(2\sqrt{2}\sqrt{h(e^x)g'_1(x)\ln^{1+\varepsilon_1} g'_1(x)
}+1\Bigl)\leq\\
\leq4\mu_f(r) \left(\sqrt{2h^2(e^x)   g_1(x)\ln^{1+\varepsilon_2}g_1(x)}
\ln^{\frac{1+\varepsilon_1}{2}}\Bigl\{ h(e^x)
{g_1(x)\ln^{1+\varepsilon_2}g_1(x)}\Bigl\}+1\right)\leq\\
\leq6\mu_f(r) h(r)\sqrt{g_1(x)}\ln^{\frac{1+\varepsilon_2}2}g_1(x)
\ln^{\frac12+\varepsilon_1}\{ h(r) g_1(x)\},\\
g_1(x)=\ln g(x)\leq\ln6+\ln\{h(r)\mu_f(r)\}+\ln g_1(x)+\ln_2\{h(r)\mu_f(r)\},\\
g_1(x)\leq2\ln\{h(r)\mu_f(r)\}.
\end{gather*}
Now for $\delta>2(\varepsilon _1+\varepsilon _2)$ we have
\begin{gather}\nonumber
 G_f(r)\leq\mu_f(r) h(r)\ln^{1/2}\{h(r)\mu_f(r)\}\times\\
 \nonumber
 \times\big(\ln_2\{h(r)\mu_f(r)\}\ln\{ h(r)\ln\{h(r)\mu_f(r)\}\}\big)^{\frac{1+\delta}2},\\
\label{6}
 M_f(r)\!\leq\! G_f(r)\!\leq\!\mu_f(r) h(r)\ln^{1/2}\{h(r)\mu_f(r)\}\ln^{1/2+\delta} h(r)
\ln_2^{1+\delta}\{h(r)\mu_f(r)\},\\
\nonumber
g_1(x)=(1+o(1))\ln\{h(r)\mu_f(r)\},\ r\to1-0,\ (r\notin E).
\end{gather}

If we choose  $\varphi_i(x)=(x+2)\ln^{1+\varepsilon_0/2}(2+x),
\ i\in\{1,2\}$ in Lemma \ref{l3}, then we get that  outside a set of finite $h$-measure
\begin{gather*}
A(r)\leq h(r)\varphi(g_1(\ln r))\leq h(r) g_1(\ln r)
\ln^{1+\varepsilon_0}g_1(\ln r)\leq\\
\leq h(r)\ln\{h(r)\mu_f(r)\}\ln_2^{1+\varepsilon }\{h(r)\mu_f(r)\}.\\
B^2(r)\leq h(r)\varphi_2(h(r)\varphi_1(g_1(\ln r)))\leq\\
 \leq h(r) h(r)\varphi_1(g_1(\ln r))\ln^{1+\varepsilon _0}( h(r)\varphi_1(g_1(\ln r)))\leq\\
\leq h^2(r)g_1(\ln r)\ln^{1+\varepsilon _0}g_1(\ln r)
\ln^{1+\varepsilon _0}\{ h(r) g_1(\ln r)\ln^{1+\varepsilon _0}g_1(\ln r)\}\leq\\
\leq h^{2+\varepsilon }(r)\ln^{1+\varepsilon }\{h(r)\mu_f(r)\}.
\end{gather*}
\end{proof}

\section{Classes of analytic functions in which the
Wiman-Valiron type inequality (\ref{2}) can be almost surely
improved}

In the sequel, the notion ``almost surely'' will be used in the
sense that the corresponding property holds almost everywhere with
respect to Lebesgue measure on the real line. Here we  prove
the following theorem.
\begin{Theorem} \label{t1}
If $f(z,t)$ is an analytic function of the form  (\ref{3}) and
a sequence $(\theta_n)_{n\geq0}$ satisfies condition (\ref{4}), then
for all $\delta>0$ and almost surely for $t$ there exists a set
$E(\delta,t)\subset(0,1)$ such that
$h\mbox{-meas}(E(\delta,t))<+\infty$ and the maximum modulus
$M_f(r,t)=M_{f_t}(r)=\max\limits_{|z|\leq r}|f_t(z)|$ satisfies the inequality
\begin{equation}\label{11}
M_f(r,t)\leq\mu_f(r) \sqrt{
h(r)}\ln^{1/4}\{h(r)\mu_f(r)\}\ln^{3/4+\delta}
h(r)\ln^{1+\delta}_2\{h(r)\mu_f(r)\}
\end{equation}
for $r\in(0,1)\backslash E(\delta,t).$
\end{Theorem}

We note that from inequality (\ref{11}) it follows that
\begin{equation}\label{12}
\varlimsup_{\begin{substack} {r\to1-0 \\ r\notin E}\end{substack}}\Delta_{h}(r,f_t)=\varlimsup_{r\to1-0}\frac{\ln M_f(r,t)-\ln\mu_f(r)}{2\ln h(r)+\ln_2\{h(r)\mu_f(r)\}}\leq\frac14.
\end{equation}
\begin{proof}[Proof.]\ Let $(\Omega, \mathcal{A}, P)$\ be
a probability space which contains a random variable
$X=X(\omega)\colon\Omega\to\mathbb{Z}_+$\ with the distribution  $
P(X=n)=|a_n|r^n/ G_f(r).$  Using Markov's inequality for the random
variable $X$  with mean value $MX=A(r)$ we get
$$
\sum_{n\geq C}\frac{|a_n|r^n}{G_f(r)}=P(X\geq
C)\leq\frac{MX}{C}=\frac{A(r)}{C}.
$$

Let $C=C(r)=A(r) h(r) \ln^{1/2+\delta}\{h(r)\mu_f(r)\}$ and $$
C_1(r)=h^2(r)\ln^2\{h(r)\mu_f(r)\}.
$$

By Lemma \ref{l4} $C_1(r)>C(r)$ for $r\in(r_0,1)\backslash E.$
Using (\ref{6}) we have
\begin{gather}
\nonumber
\sum_{n\geq C_1(r)}|a_n|r^n\leq\sum_{n\geq C(r)}|a_n|r^n\leq
\frac{A(r) G_f(r)}{A(r) h(r)\ln^{1/2+\delta}\{h(r)\mu_f(r)\}}\leq\\
\label{7}
\leq\frac{ h(r)\mu_f(r)\ln^{1/2+\delta} h(r)\ln^{1/2+\delta}\{h(r)\mu_f(r)\}}{ h(r)\ln^{1/2+\delta}\{h(r)\mu_f(r)\}}=
\mu_f(r)\ln^{1/2+\delta} h(r)
\end{gather}
for $r\notin E,$ where $E$ is a set of finite $h$-measure.

 We put $k(r)= h(r)\mu_f(r)$ in Lemma \ref{l2} and let $(r_k)_{k\geq0}$ be the sequence
for which conse\-quences of this lemma are valid.
We denote by $F_k$ the set of $t\in\mathbb{R}$ such that
$$
W(r_k)=\max_{0\leq \psi\leq2\pi}\Biggl|\sum_{n\leq[C_1(r_k)]}
a_nr_k^ne^{in\psi}e^{i\theta_nt}\Biggl|\geq A_{\beta q}S_{[C_1(r_k)]}(r_k)\ln^{1/2}[C_1(r_k)].
$$
It follows from Lemma \ref{l2} with $\beta=2$ that
$$
\sum_{k=1}^{+\infty}P(F_k)\leq
\sum_{k=1}^{+\infty}\frac1{[C_1(r_k)]^2}\leq
\sum_{k=1}^{+\infty}\frac1{[\ln\{\mu_f(r_k)h(r_k)\}]^2}\leq
\sum_{k=1}^{+\infty}\frac4{k^2}<+\infty.
$$

Then by Borel-Cantelli's lemma for  $k\geq k_0(t)$ and almost
surely for $t\in\mathbb{R}$ we obtain
\begin{equation}\label{8}
W(r_k)<A_qS_{[C_1(r_k)]}(r_k)\ln^{1/2}[C_1(r_k)].
\end{equation}
 From  inequalities (\ref{6}), (\ref{8}) and $S_{[C_1(r)]}(r)\leq M_f(r)\mu_f(r)$
 it follows that
 \begin{gather}
 \nonumber
 W(r_k)<\sqrt{\mu_f(r_k)}\sqrt{\mu_f(r_k)h(r_k)}\ln^{1/4}\{h(r_k)\mu_f(r_k)\}\times\\
 \nonumber
 \times
 \ln^{1/4+2\delta/3}h(r_k)
 \ln^{1/2+2\delta/3}_2\{h(r_k)\mu_f(r_k)\}
 \ln^{1/2}(h^2(r)\ln^2\{h(r)\mu_f(r)\})\leq\\
 \label{9}
 \leq\mu_f(r_k)\sqrt{h(r_k)}\ln^{1/4}\{h(r_k)\mu_f(r_k)\}
 \ln^{3/4+3\delta/4}h(r_k)
 \ln^{1+3\delta/4}_2\{h(r_k)\mu_f(r_k)\}.
 \end{gather}

 Since
 $$
 M_f(r,t)\leq\sum_{n\geq C_1(r)}|a_n|r^n+W(r),
 $$
 from (\ref{7}) and (\ref{9}) we get
 \begin{gather}
 \nonumber
 M_f(r_k,f)\leq\mu_f(r_k)\sqrt{h(r_k)}\ln^{1/4}\{h(r_k)\mu_f(r_k)\}\times\\
 \label{10}\times
 \ln^{3/4+{4\delta}/5}h(r_k)
 \ln^{1+{4\delta}/5}_2\{h(r_k)\mu_f(r_k)\}.
 \end{gather}
 We suppose that $r_{k_2(t)}\in(0,1)$ is some number outside the set $E.$
 Then for $r\in(r_p,r_{p+1}),$ $p>k_2(t)$ by Lemma \ref{l2} we obtain
 \begin{gather}
 \label{st1}
 \mu_f(r_{p+1})h(r_{p+1})\leq e\mu_f(r_{p})h(r_{p})\leq e\mu_f(r)h(r),\\
 \label{st2}
 \mu_f(r_{p+1})=h(r_{p+1})\frac{\mu_f(r_{p+1})}{h(r_{p+1})}\leq
 eh(r_{p})\frac{\mu_f(r_{p})}{h(r_{p+1})}\leq
 eh(r)\frac{\mu_f(r)}{h(r_{p+1})}\leq e\mu_f(r),\\
 \label{st3}
 h(r_{p+1})=\frac{\mu_f(r_{p+1})h(r_{p+1})}{\mu_f(r_{p+1})}\leq
 e\frac{\mu_f(r)h(r)}{\mu_f(r_{p+1})}\leq e h(r).
 \end{gather}

 Finally, from (\ref{10}) we have for $r\in(r_p,r_{p+1})$
 \begin{gather*}
 M_f(r,t)\leq M_f(r_{p+1},t)\leq\\
\leq\mu_f(r) \sqrt{ h(r)}\ln^{1/4}\{h(r)\mu_f(r)\}\ln^{3/4+\delta} h(r)\ln^{1+\delta}_2\{h(r)\mu_f(r)\}
 \end{gather*}
almost surely for $t\in\mathbb{R}$.
\end{proof}
By $\mathcal{L}$ we denote the class of increasing to $+\infty$
functions $l(x)$ on $[0,+\infty).$ Let
$$
\gamma(l)=\varlimsup_{x\to+\infty}\frac{\ln l(x)}{\ln x}.
$$

Now we consider the class of analytic functions of the form  (\ref{3}), for which
the sequence $(\theta_n)_{n\geq0}$ satisfies the condition
\begin{equation}\label{13}
\frac{\theta_{n+1}}{\theta_n}\geq1+\frac1{\varphi(n)},\ \varphi\in \mathcal{L}.
\end{equation}

{\it What constant  can we put in the inequality (\ref{12}) instead of $1/4$
for this class of analytic functions? Under which conditions on the function $\varphi(x)$ does
the inequality (\ref{12}) hold?}
We give answers to these questions in Corollaries \ref{c1} and~\ref{c2}.

Firstly, we note that  one cannot  sharpen  inequality  (\ref{2}) for a rapidly growing function $\varphi(x)$.
Indeed, if $\varphi(x)=x,$ then we may choose $\theta_n=n,\ h(r)=(1-r)^{-1}$
and  $g(z)=\sum_{n=0}^{+\infty}e^{\sqrt{n}}z^n.$
As it is known from \cite{sul},
\begin{gather*}
M_g(r,t)=\max\{|g(r,t)|\colon |z|\leq r\}=
\max_{0\leq\psi\leq2\pi}\Biggl|\sum_{n=0}^{+\infty}a_nr^ne^{int}e^{in\psi}\Biggl|=\\
=\max_{0\leq\psi\leq2\pi}\Biggl|\sum_{n=0}^{+\infty}a_nr^ne^{in(t+\psi)}\Biggl|=
\max_{0\leq\psi\leq2\pi}\Biggl|\sum_{n=0}^{+\infty}a_nr^ne^{in\psi}\Biggl|=\\
=M_g(r)\geq C_1\mu_g(r) h(r)\ln^{1/2}\{\mu_g(r) h(r)\},
\end{gather*}
when $r\to1-0$ and $t\in\mathbb{R}.$
So, in order to improve inequality (\ref{2}) $\varphi(x)$ must satisfy
the  condition $\gamma(\varphi)<1.$
\begin{Theorem} \label{t2}
Let $f_t(z)$ be an analytic function of the form  (\ref{3}), $h\in
H,$ à sequence $(\theta_n)_{n\geq0}$ satisfy condition
(\ref{13}), where $\varphi\in\mathcal{L}.$ If $v\in\mathcal{L}$
and $\gamma(v)\leq1/4,$ then almost surely
for $t\in\mathbb{R},$
 all $\varepsilon >0$ there exists a set $E(\varepsilon,t)\subset(0,1)$
 such that $h\mbox{-meas}(E(\varepsilon ,t))<+\infty$ and for
 $r\in(0,1)\backslash E(\varepsilon ,t)$ we have
\begin{gather}
\nonumber
M_f(r,t)\leq\sqrt{ h(r)\ln h(r)}\mu_f(r)\ln^{1/4}\{h(r)\mu_f(r)\}\ln^{1+\varepsilon }
\{\ln h(r)\ln\{h(r)\mu_f(r)\}\}\times\\
\label{17}
\times\Biggl(v\Bigl(8h^2(r)\ln\{h(r)\mu_f(r)\}\Bigl)
+\varphi^{\frac12}\Biggl(
\frac{h^{\frac32}(r)\ln^{\frac54}\{h(r)\mu_f(r)\}\ln^{1+\varepsilon }_2\{h(r)\mu_f(r)\}}{v( h(r)\ln\{h(r)\mu_f(r)\})}\Biggl)\Biggl).
\end{gather}
\end{Theorem}

In order to prove this theorem we need a lemma from \cite{filMS6}.
\begin{Lemma}[\cite{filMS6}] \label{l5}
If $(\theta_n)_{n\geq0}$ satisfies condition (\ref{13}),
then for all  $\beta>0$
$$
P\Biggl(\max_{0\leq\psi\leq2\pi}\Biggl|\sum_{k=1}^{N}
a_ke^{ik\psi}e^{i\theta_kt}\Biggl|\geq A_{\beta}
\Bigl\{\varphi(N)S_N\ln N\Bigl\}^{1/2}\Biggl)\leq\frac1{N^{\beta}},
$$
where $A_{\beta}$ is a constant which depends only on $\beta.$
\end{Lemma}
\begin{proof}[Proof of Theorem \ref{t2}.]
By Lemma \ref{l4} we obtain outside a set of finite $h$-measure
\begin{equation}\label{14}
A(r)\leq h(r)\ln\{h(r)\mu_f(r)\}\ln^{1+\varepsilon}_2\{h(r)\mu_f(r)\}.
\end{equation}

We put $C(r)=A(r)T(r),$ where
$$
T(r)=\frac{\sqrt{h(r)}\ln^{1/4}\{h(r)\mu_f(r)\}}{v(h^2(r)\ln\{h(r)\mu_f(r)\})}.
$$
Then from (\ref{14}) we have
\begin{gather*}
C(r)=A(r)T(r)\leq h(r)\ln\{h(r)\mu_f(r)\}\ln^{1+\varepsilon }_2\{h(r)\mu_f(r)\} T(r)=\\
=\frac{h^{3/2}(r)\ln^{5/4}\{h(r)\mu_f(r)\}\ln^{1+\varepsilon }_2\{h(r)\mu_f(r)\} }{v(h^2(r)\ln\{h(r)\mu_f(r)\})}=C_1(r).
\end{gather*}

 Now using Markov's inequality we get
 \begin{gather}
 \nonumber
 \sum_{n\geq C_1(r)}|a_n|r^n\leq\sum_{n\geq C(r)}|a_n|r^n\leq
 \frac{ G_f(r)}{T(r)}\leq\\
 \nonumber
 \leq\frac{ h(r)\mu_f(r)\{\ln h(r)\ln\{h(r)\mu_f(r)\}\}^{1/2}\ln^{1+\delta}\{\ln h(r)\ln\{h(r)\mu_f(r)\}\}}
 {\sqrt{ h(r)}\ln^{1/4}\{h(r)\mu_f(r)\}}\times\\
 \nonumber
 \times v(h^2(r)\ln\{h(r)\mu_f(r)\})=\\
 \nonumber
 =\mu_f(r)\sqrt{ h(r)\ln h(r)}\ln^{1/4}\{h(r)\mu_f(r)\}\ln^{1+\delta}\{\ln h(r)\ln\{h(r)\mu_f(r)\}\}\times\\
 \label{15}
 \times v(h^2(r)\ln\{h(r)\mu_f(r)\}).
 \end{gather}

Let $k(r)= h(r)\mu_f(r)$ and $(r_k)_{k\geq0}$  be the sequence
for which consequences of Lemma \ref{l2} are valid.
Denote by $G_k$ the set of such $t\in\mathbb{R},$ for which
 \begin{gather*}
 W_1(r_k)=\max_{0\leq \psi\leq2\pi}\Biggl|\sum_{n\leq[C_1(r_k)]}
a_nr_k^ne^{in\psi}e^{i\theta_nt}\Biggl|\geq\\ \geq A_{\beta}\Bigl(\varphi([C_1(r_k)])S_{[C_1(r_k)]}(r_k)\ln[C_1(r_k)]\Bigl)^{1/2},
 \end{gather*}
where $ S^2_f(r)=\sum_{n=0}^{+\infty}|a_n|^2r^{2n}.$

Since $\gamma(v)\leq1/4,$ we have
\begin{gather*}
C_1(r)>\frac{h^{3/2}(r)\ln^{5/4}\{h(r)\mu_f(r)\}\ln^{1+\varepsilon }_2\{h(r)\mu_f(r)\}}{(h^2(r)\ln\{h(r)\mu_f(r)\})^{1/4}}>\\
>{ h(r)\ln\{h(r)\mu_f(r)\}}\ln^{1+\varepsilon }_2\{h(r)\mu_f(r)\}>{\ln\{h(r)\mu_f(r)\}}.
\end{gather*}
So, by Lemma \ref{l2} $\ln k(r_n)>n/2,$ i.e.
$\ln\{h(r_n)\mu_f(r_n)\}>n/2.$ Then
$$
C_1(r_n)>{\ln\{h(r_n)\mu_f(r_n)\}}>{n/2}.
$$

Using Lemma \ref{l5} with $\beta=2$ we get
$$
\sum_{k=1}^{+\infty}P(G_k)<
\sum_{k=1}^{+\infty}\frac{1}{N^{\beta}(r_k)}<
\sum_{k=1}^{+\infty}\frac4{k^2}<+\infty.
$$
Now by Borel-Cantelli's lemma  for $k\geq k_2(t)$ and almost surely
$t\in\mathbb{R}$ we obtain
$$
W_1(r_k)<A_{\beta}\Bigl(\varphi([C_1(r_k)])S_{[C_1(r_k)]}(r_k)\ln[C_1(r_k)]\Bigl)^{1/2}.
$$
Using the inequality $ S^2_f(r)\leq G_f(r)\mu_f(r),$ we obtain
\begin{gather}
\nonumber
W_1(r_k)<\sqrt{h(r_k)\ln h(r_k)}\mu_f(r_k)\ln^{1/4}\{h(r_k)\mu_f(r_k)\}\times\\
\nonumber
\times\ln^{1+\delta}\{\ln h(r_k)\ln\{h(r_k)\mu_f(r_k)\}\}\times\\
\label{16}
\times
\varphi^{1/2}\Biggl(\frac{h^{3/2}(r_k)\ln^{5/4}\{h(r_k)\mu_f(r_k)\}\ln_2^{1+\delta}\{h(r_k)\mu_f(r_k)\}}{v(h^2(r_k)\ln\{h(r_k)\mu_f(r_k)\})}\Biggl).
\end{gather}
It follows from (\ref{16}) and (\ref{17}) that
\begin{gather*}
M_f(r_k,t)\leq\sqrt{h(r_k)\ln h(r_k)}\mu_f(r_k)\ln^{1/4}\{h(r_k)\mu_f(r_k)\}
\times\\
\times\ln^{1+\delta}\{\ln h(r_k)\ln\{h(r_k)\mu_f(r_k)\}\}
\Biggl(v(h^2(r_k)\ln\{h(r_k)\mu_f(r_k)\})+\\
+\varphi^{1/2}
\Biggl(\frac{h^{3/2}(r_k)\ln^{5/4}\{h(r_k)\mu_f(r_k)\}\ln_2^{1+\delta}\{h(r_k)\mu_f(r_k)\}}
{v(h^2(r_k)\ln\{h(r_k)\mu_f(r_k)\})}\Biggl)\Biggl).
\end{gather*}

 Using (\ref{st1})--(\ref{st3}) we get for $r\in(r_p,r_{p+1})$
\begin{gather*}
M_f(r,t)\leq
\sqrt{ h(r)\ln h(r)}\mu_f(r)\ln^{1/4}\{h(r)\mu_f(r)\}\times\\
\times
\ln^{1+2\delta}\{\ln h(r)\ln\{h(r)\mu_f(r)\}\}
\Biggl(v(8h^2(r)\ln\{h(r)\mu_f(r)\})+\\
+\varphi^{1/2}
\Biggl(\frac{h^{3/2}(r)\ln^{5/4}\{h(r)\mu_f(r)\}\ln^{1+2\delta}\ln\{h(r)\mu_f(r)\}}
{v(h^2(r)\ln\{h(r)\mu_f(r)\})}\Biggl)\Biggl).
\end{gather*}
\end{proof}

In the case when $\ln\varphi(x)=o(\ln_2 x),\ x\to+\infty$
 we have the following corollary.
 \begin{Corollary} \label{c1}
Let $f_t(z)$ be an analytic function of the form  (\ref{3}), $h\in
H,$ a~sequence $(\theta_n)_{n\geq0}$ satisfy condition
(\ref{13}), where $\varphi\in\mathcal{L}$ and
$\ln\varphi(x)=O(\ln_2 x),$ $x\to+\infty.$ Then there exists a
set $E(\delta,t)\subset(0,1)$ such that
$h\mbox{-meas}(E(\delta,t))<+\infty$ and almost surely
for $t\in\mathbb{R}$ we get
 $$
 \varlimsup_{\begin{substack} {r\to1-0 \\ r\notin E}\end{substack}}\Delta_{h}(r,f_t)\leq\frac14.
 $$
 \end{Corollary}
\begin{Corollary} \label{c2}
Let $f_t(z)$ be an analytic function of the form  (\ref{3}),
$h\in H,$ a sequence $(\theta_n)_{n\geq0}$ satisfy condition (\ref{13}),
where $\varphi\in\mathcal{L}$ and
\begin{equation}\label{19}
\gamma(\varphi)=\varlimsup_{n\to+\infty}
\frac{1}{\ln n}\ln\frac{\theta_n}{\theta_{n+1}-\theta_n}\leq\delta\in[0,1/2).
\end{equation}
Then for all analytic functions $f_t$ there exists a set
$E(\delta,t)\subset(0,1)$ such that
$h\mbox{-meas}(E(\delta,t))<+\infty$  and almost surely
for $t\in\mathbb{R}$ we have
 $$
 \varlimsup_{\begin{substack} {r\to1-0 \\ r\notin E}\end{substack}}\Delta_{h}(r,f_t)\leq\frac{1+3\delta}{4+2\delta}.
 $$
 \end{Corollary}
\begin{proof}[Proof.]
If $\gamma(\varphi)=\delta\in[0,1/2),$ then we may choose
$v(x)=x^{\alpha}, \ \alpha\in[0,1/4).$ So,
\begin{gather}
\nonumber
\ln v(8h^2(r)\ln\{h(r)\mu_f(r)\})=(\alpha+o(1))(2\ln h(r)+\ln_2\{h(r)\mu_f(r)\}),\\
\nonumber
\ln \Biggl(\varphi^{1/2}\Biggl(
\frac{h^{3/2}(r)\ln^{5/4}\{h(r)\mu_f(r)\}\ln^{1+\varepsilon }_2\{h(r)\mu_f(r)\}}{v(h^2(r)\ln\{h(r)\mu_f(r)\})}\Biggl)\Biggl)\leq\\
\nonumber
\leq\!(1\!+\!o(1))\Bigl(\frac{3\delta}{4}\ln h(r)\!+\!\frac{5\delta}{8}\ln_2\{h(r)\mu_f(r)\}\!-\!
\delta\alpha\ln h(r)\!-\!\frac{\delta\alpha}{2}\ln_2\{h(r)\mu_f(r)\}\Bigl)\!=\\
\nonumber
=\Bigl(\frac{3\delta}{8}-\frac{\delta\alpha}{2}+o(1)\Bigl)
2\ln h(r)+
\Bigl(\frac{5\delta}{8}-\frac{\delta\alpha}{2}+o(1)\Bigl)
\ln_2\{h(r)\mu_f(r)\}\leq\\
\label{tri}
\leq\Bigl(\frac{5\delta}{8}-\frac{\delta\alpha}{2}+o(1)\Bigl)
(2\ln h(r)+\ln_2\{h(r)\mu_f(r)\}).
\end{gather}

From the equation $\alpha=\frac{5\delta}{8}-\frac{\delta\alpha}{2}$ we may determine
 $\alpha=\frac{5\delta}{4(2+\delta)}$  and get as $r\to1-0$
\begin{gather*}
\ln M_f(r,t)\leq(1+o(1))\Bigl(\frac12\ln h(r)+\ln\mu_f(r)+\\
+\frac14\ln_2\{h(r)\mu_f(r)\}+
\alpha(\ln h(r)+\ln_2\{h(r)\mu_f(r)\})\Bigl).
\end{gather*}

Therefore,
\begin{gather*}
 \varlimsup_{\begin{substack} {r\to1-0 \\ r\notin E}\end{substack}}\Delta_{h}(r,f_t)=\varlimsup_{\begin{substack} {r\to1-0 \\ r\notin E}\end{substack}}\frac{\ln M_f(r,t)-\ln\mu_f(r)}{2\ln h(r)+\ln_2\{h(r)\mu_f(r)\}}\leq\\
\leq \frac14+\alpha=\frac14+\frac{5\delta}{4(2+\delta)}=\frac{1+3\delta}{4+2\delta}.
\end{gather*}
\end{proof}

So, we can improve inequality (\ref{2})
for all analytic functions of the form  (\ref{3})
 and all $h\in H,$ when
$\gamma(\varphi)<1/2.$ As it is noted above, this
inequality cannot be improved if $\gamma(\varphi)\geq1.$
 Can we improve inequality (\ref{2}) for all analytic functions of the form  (\ref{3})
 by condition $\gamma(\varphi)<1$?

Corollary \ref{c3} gives a positive answer to this question by some choice of the function~$ h(r).$
\begin{Corollary} \label{c3}
Let $f_t(z)$ be an analytic function of the form  (\ref{3}),
$h\in H\colon$ $\ln_2\mu_f(r)=o(\ln h(r)),$ $ r\to1-0,$ a sequence $(\theta_n)_{n\geq0}$ satisfy condition (\ref{13}),
where $\varphi\in\mathcal{L}$ and \begin{equation}\label{qvad}
\gamma(\varphi)=\varlimsup_{n\to+\infty}
\frac{1}{\ln n}\ln\frac{\theta_n}{\theta_{n+1}-\theta_n}\leq\delta\in[0,1).
\end{equation}
Then for all analytic functions $f_t$ there  exists a set $E(\delta,t)\subset(0,1)$ such that $h\mbox{-meas}(E(\delta,t))<+\infty$ and
 almost surely for  $t\in\mathbb{R}$
 $$
 \varlimsup_{\begin{substack} {r\to1-0 \\ r\notin E}\end{substack}}\Delta_{h}(r,f_t)\leq\frac{1+2\delta}{4+2\delta}.
 $$
 \end{Corollary}
\begin{proof}[Proof.]
It follows from (\ref{tri}) that
\begin{gather*}
\ln \Biggl(\varphi^{1/2}\Biggl(
\frac{h^{3/2}(r)\ln^{5/4}\{h(r)\mu_f(r)\}\ln^{1+\varepsilon }_2\{h(r)\mu_f(r)\}}{v(h^2(r)\ln\{h(r)\mu_f(r)\})}\Biggl)\Biggl)\leq\\
\leq\Bigl(\frac{3\delta}{8}-\frac{\delta\alpha}{2}+o(1)\Bigl)
2\ln h(r)+
\Bigl(\frac{5\delta}{8}-\frac{\delta\alpha}{2}+o(1)\Bigl)
\ln_2\{h(r)\mu_f(r)\}\leq\\
\leq\Bigl(\frac{3\delta}{8}\!-\!\frac{\delta\alpha}{2}\!+\!o(1)\Bigl)
2\ln h(r)\!\leq\!
\Bigl(\frac{3\delta}{8}\!-\!\frac{\delta\alpha}{2}+o(1)\Bigl)
(2\ln h(r)\!+\!\ln_2\{h(r)\mu_f(r)\}).
\end{gather*}

From the equation $\alpha=\frac{3\delta}{8}-\frac{\delta\alpha}{2}$
we determine  $\alpha=\frac{3\delta}{4(2+\delta)}$ and
\begin{gather*}
 \varlimsup_{\begin{substack} {r\to1-0 \\ r\notin E}\end{substack}}\Delta_{h}(r,f_t)=\varlimsup_{\begin{substack} {r\to1-0 \\ r\notin E}\end{substack}}\frac{\ln M_f(r,t)-\ln\mu_f(r)}{2\ln h(r)+\ln_2\{h(r)\mu_f(r)\}}\leq\\
\leq \frac14+\alpha=\frac14+\frac{3\delta}{4(2+\delta)}=\frac{1+2\delta}{4+2\delta}.
\end{gather*}
\end{proof}

\section{Baire's categories  and Wiman-Valiron's \break type inequality  for analytic functions}

Let $h\in H$ and
$\theta=(\theta_n)_{n\geq 0}$ be a fixed sequence satisfying
condition \eqref{13}, such that $\gamma(\varphi)\leq\delta.$ Similarly to [11], we define the following sets
\begin{gather*}
F_{1h}(f,\theta, E)=\Big\{t\in\mathbb{R}\colon
\varlimsup_{\begin{substack} {r\to1-0 \\ r\notin E}\end{substack}}\Delta_h(r,f_t)\leq\frac{1+3\delta}{4+2\delta}\Big\}\\
F_{2h}(f,\theta)=\Big\{t\in\mathbb{R}\colon \Bigl(\forall\ \eta
>\frac{1+3\delta}{4+2\delta}\Bigl) \big[
h\mbox{-meas}(E(\eta, f_t, h))<+\infty\big]\Big\}\\
F_{3h}(f,\theta)=\Big\{t\in\mathbb{R}\colon
\varliminf_{r\to1-0}\Delta_h(r,f_t)\leq\frac{1+3\delta}{4+2\delta}\Big\},\\
F_{4h}(f,\theta)=\Big\{t\in\mathbb{R}\colon
\varliminf_{r\to1-0}\Delta_h(r,f_t)\leq\frac{1+2\delta}{4+2\delta}\Big\}.
\end{gather*}

By Corollary \ref{c2} we conclude that for analytic functions in
$\mathbb{D}$ there exists a set $E(f)$ of finite $h$-measure
such that the set $F_{1h}(f,\theta)$ is ``large'' in the sense of Lebesgue measure.
Therefore, we obtain some information on the sets $F_{2h}(f,\theta),\ F_{3h}(f,\theta).$

Similarly to \cite{filBai}, the following { question} arises naturally:
{\it does there exist a
 set $E=E(f)$ of the
finite $h$-measure such that the set $F_{1h}(f,\theta, E)$ is
residual in $\mathbb{R}$ for every analytic function $f$?}

We recall that a set
$B\subset\mathbb{R}$ is called residual in $\mathbb{R},$ if its
complement $\overline{B}=\mathbb{R}\backslash B$ is a set of the first
Baire category  in~$\mathbb{R}$. It is clear, that if the answer
to the question is affirmative, then the sets $F_{2h}(f,\theta),\
F_{3h}(f,\theta)$ are residual in~${\mathbb R}$. However for some analytic functions the set
$F_{1h}(f,\theta, E)$ is a set of  the first Baire category (see similar assertion for the entire function $f(z)=e^z$ in~\cite{filBai}). It
follows from the following theorem.

\begin{Theorem} \label{t3}
Let a sequence $(\theta_n)_{n\geq0}$ satisfy condition (\ref{4}),
$f(z)\!=\!\sum\limits_{n=0}^{+\infty}e^{n^{\varepsilon}}z^n\!,$ $\varepsilon \in(0,1),$ and $h(r)=(1-r)^{-1}.$
Then there exists a constant $C=C(\theta,\varepsilon )>0$ such that for all sequences
$(r_n)_{n\geq0}$ increasing to 1 the set $$
F_3=\Bigl\{t\in\mathbb{R}\colon
\varlimsup_{n\to+\infty}\frac{M_{f_t}(r_n)}
{h(r_n)\mu_f(r_n)\ln^{1/2}\{h(r_n)\mu_f(r_n)\}}\leq C\Bigl\}
$$
is a set of the first Baire category.
\end{Theorem}

We need the following lemma from \cite{kahww}.
\begin{Lemma}[\cite{kahww}] \label{l6}
For every $q>1$ there exist positive constants $A=A(q)$ and
$B=B(q)$ such that for each interval $I\subset{\mathbb R}$ and
every trigonometrical polynomial
$Q(t)=\sum_{n=1}^Nc_ne^{i\lambda_nt},$\
$0<\lambda_1<\lambda_2<\ldots <\lambda_N,$\ for which
$|I|\ge\frac{B}{\lambda_1}>0$ and
$\frac{\lambda_{n+1}}{\lambda_n}\ge q$, $1\le n\le N-1$, there
exists a point $t_0\in I$ such that
$$
\mathop{\rm Re}Q(t_0)\ge A\sum_{n=1}^N|c_n|.
$$
\end{Lemma}
\begin{proof}[Proof of Theorem \ref{t3}.]
For the function $f(z)=\sum\limits_{n=1}^{+\infty}\exp\{n^{\varepsilon }\}z^n,$ $\varepsilon \in(0,1)$
(see [3]) there exists $C_0(\varepsilon )\in(0,1)$ such that we have
\begin{gather*}
C_0^{-1}(\varepsilon )\frac{\mu_f(r)}{1-r}
\geq\frac{ M_f(r)}{\sqrt{\ln M_f(r)}}
\geq C_0(\varepsilon )\frac{\mu_f(r)}{1-r}, \ r\to1-0.
\end{gather*}

Then we obtain as $r\to1-0$
\begin{gather}
\nonumber
\ln M_f(r)-\frac12\ln_2 M_f(r)\geq
\ln C_0(\varepsilon )+\ln\frac{\mu_f(r)}{1-r},\
\ln M_f(r)\geq\ln\frac{\mu_f(r)}{1-r},\\
\label{20}
 M_f(r)\geq C_0(\varepsilon )\frac{\mu_f(r)}{1-r}\ln^{1/2}\frac{\mu_f(r)}{1-r}.
\end{gather}

Let $(r_n)_{n\geq0}$ be some sequence  increasing to 1.
We put
$$
q=\inf\{\theta_{n+1}/\theta_n\colon n\geq0\}>1,
$$
$A=A(q)$ and $B=B(q)$ are the constants from Lemma~\ref{l6}, $C(\varepsilon )=AC_0(\varepsilon ).$
We consider a sequence of  numbers $(C_n(\varepsilon ))_{n\geq0}$
increasing to $C(\varepsilon )$. Define the set
\begin{gather*}
F_{mk}=\bigg\{\!t\in\mathbb{R}\colon\ (\forall l\geq k)\! \bigg[{M_{f_t}(r_l)}
\leq C_m(\varepsilon)
{\frac{\mu_f(r_l)}{1-r_l}\ln^{1/2}\frac{\mu_f(r_l)}{1-r_l}}\bigg]
  \bigg\}.
\end{gather*}
where the integers
$k\geq0,\ m\geq0$ are fixed.

For fixed $r\in(0,1)$ we consider the function
$$
\alpha(t,\varphi)=\Biggl|\sum_{n=0}^{+\infty} \exp\{i\theta_nt+n^{\varepsilon }+in\varphi\}r^n\Biggl|,
$$
which is continuous in $\mathbb{R}^2$ and periodic in the variables
$t$ and $\varphi.$
Then the function $\beta(t)=\max_{\varphi}\alpha(t,\varphi)=M_f(r,t)$
is continuous at every point $t\in\mathbb{R}.$
We remark, that the set $F_{mk}$ is closed in $\mathbb{R}.$

Now we prove that the set $\overline{F_{mk}}$ is everywhere dense.
Consider an arbitrary interval $I\subset\mathbb{R}, \ |I|>0$
and show that it contains some point $t_0$ from the set~$\overline{F_{mk}}.$

Let us choose $p\geq1, \delta>0$ such that
\begin{equation}\label{23}
|I|\ge\frac{B}{\theta_p},\ 1-2\delta>\sqrt{\frac{C_m(\varepsilon)}{C(\varepsilon)}}.
\end{equation}

Using (\ref{20}), we may define
\begin{gather*}
x_1\!=\!x_1(\varepsilon)\!=\!\inf\Biggl\{r\in(0,1)\colon
\sum_{n=0}^{+\infty}\exp\{n^{\varepsilon }\}r^n\!
\geq\!(1-2\delta)C_0(\varepsilon )
\frac{\mu_f(r)}{1-r}\ln^{1/2}\frac{\mu_f(r)}{1-r}\Biggl\},\\
x_2=x_2(\varepsilon )=\inf\Biggl\{r\in(0,1)\colon
\sum_{n=0}^{p}\exp\{n^{\varepsilon }\}r^n\leq\frac{A}{A+1}
\delta\sum_{n=0}^{+\infty}\exp\{n^{\varepsilon }\}r^n\Biggl\}.
\end{gather*}

Now choose  integers $l\geq k$ and $s>p$ such that the following inequalities
\begin{gather}\label{24}
r_l>\max\{x_1,x_2\},\ \sum_{n=s+1}^{+\infty}\exp\{n^{\varepsilon }\}r_l^n
\leq
\frac{A}{A+1}
\delta\sum_{n=0}^{+\infty}\exp\{n^{\varepsilon }\}r_l^n
\end{gather}
hold.

By Lemma \ref{l2} there exists a point $t_0$ in the interval $I$ such that
\begin{equation}\label{25}
 \mathop{\rm Re}\Biggl(\sum_{n=p}^se^{i\theta_n t_0+n^{\varepsilon }}r_l^n\Biggl)\geq
 A\sum_{n=p}^s\exp\{n^{\varepsilon }\}r_l^n.
 \end{equation}
Using the definitions of $x_1, x_2$ from (\ref{20})--(\ref{25}) we deduce
\begin{gather*}
M_f(r_l,t_0)=\max_{\varphi}|f_{t_0}(r_le^{i\varphi})|\geq\\
\geq|f_{t_0}(r_l)|\geq\mathop{\rm Re}f_{t_0}(r_l)
\geq\mathop{\rm Re}\Biggl(\sum_{n=p}^s\exp\{i\theta_n t_0+n^{\varepsilon }\}r_l^n\Biggl)-\\
-\sum_{n\notin[p,s]}\exp\{n^{\varepsilon }\}r_l^n
\geq A\sum_{n=p}^s\exp\{n^{\varepsilon }\}r_l^n-
\sum_{n\notin[p,s]}\exp\{n^{\varepsilon }\}r_l^n=\\
=A\sum_{n=0}^{+\infty}\exp\{n^{\varepsilon }\}r_l^n
-(1+A)\sum_{n\notin[p,s]}\exp\{n^{\varepsilon }\}r_l^n\geq\\
\geq A\sum_{n=0}^{+\infty}\exp\{n^{\varepsilon }\}r_l^n
-(1+A)\frac{2A}{1+A}\delta\sum_{n=0}^{+\infty}\exp\{n^{\varepsilon }\}r_l^n=\\
=A(1-2\delta)\sum_{n=0}^{+\infty}\exp\{n^{\varepsilon }\}r_l^n
\geq\frac{C(\varepsilon )}{C_0(\varepsilon )}(1-2\delta)^2C_0(\varepsilon )\frac{\mu_f(r_l)}{1-r_l}\ln^{1/2}\frac{\mu_f(r_l)}{1-r_l}\geq\\
\geq C_m(\varepsilon )\frac{\mu_f(r_l)}{1-r_l}\ln^{1/2}\frac{\mu_f(r_l)}{1-r_l}.
\end{gather*}

Therefore, $t_0\in\overline{F_{mk}}.$

Since the set $F_{mk}$ is closed in $\mathbb{R}$ and its complement
$\overline{F_{mk}}$ is everywhere dense,  the set $F_{mk}$ is nowhere dense.
Hence
$$
F_3=\bigcup^{+\infty}_{m=0}\bigcup^{+\infty}_{k=0}F_{mk}
$$
is a set of the first Baire category.
\end{proof}
\begin{Theorem} \label{t4}
If a sequence $(\theta_n)_{n\geq0}$ satisfies condition (\ref{13}) and $h(r)=\frac1{1-r}$,
then for every analytic function $f$ the
set $F_{3h}(f,\theta)$ is residual in $\mathbb{R}.$
\end{Theorem}
\begin{proof}[Proof.]
Let $f$ be an arbitrary analytic function in $\mathbb{D}.$
We consider the sequence $(c_n)_{n\geq0}$ such that
$$
c_n\downarrow\frac{1+3\delta}{4+2\delta}, \ n\to+\infty.
$$

Fix integers $m\geq0,\ k\geq0$ and define the set
$$
G_{mk}=\Biggl\{t\in\mathbb{R}\colon
M_f(r,t)\geq\frac{\mu_f(r)}{(1-r)^{c_m}}\ln^{c_m}\frac{\mu_f(r)}{1-r},\ \forall r>1-\frac1{k+1}\Biggl\}.
$$

As it has been proved above, for every fixed $r\in(0,1)$
the function $\beta(t)=M_f(r,t)$ is continuous at every point $t_0\in\mathbb{R}.$
Then  the set $G_{mk}$ is closed in $\mathbb{R}.$
By Corollary \ref{c2} the set $\overline{G_{mk}}$ is everywhere dense.
Therefore, $G_{mk}$ is nowhere dense and
$$
G=\bigcup^{+\infty}_{m=0}\bigcup^{+\infty}_{k=0}G_{mk}
$$
is a set of the first Baire category. So,
$F_{3h}(f,\theta)=\overline{G}$ is residual  in $\mathbb{R}.$
\end{proof}
\begin{Theorem} \label{t5}
If a sequence $(\theta_n)_{n\geq0}$ satisfies condition (\ref{13}) and $h(r)=\frac1{1-r}$,
then for all analytic functions $f$ such that $\ln_2\mu_f(r)=o(\ln(1-r)),$ $r\to1-0$,
the set $F_{4h}(f,\theta)$ is residual in $\mathbb{R}.$
\end{Theorem}


\end{document}